\documentclass{amsart}
\usepackage{latexsym,amssymb,amsmath,amsthm,amscd,graphicx}
\usepackage{esint} 

\newtheorem{theorem}{Theorem}[section]
\newtheorem{lemma}[theorem]{Lemma}

\newtheorem{proposition}[theorem]{Proposition}

\theoremstyle{definition}
\newtheorem{definition}[theorem]{Definition}
\newtheorem{example}[theorem]{Example}

\theoremstyle{remark}
\newtheorem{remark}[theorem]{Remark}

\numberwithin{equation}{section}

\begin{document}

\title[Boundedness of composition operators on Morrey spaces]
{Boundedness of composition operators on Morrey spaces and weak Morrey spaces}

\author{Naoya Hatano, Masahiro Ikeda, Isao Ishikawa, and Yoshihiro Sawano}
\address[Naoya Hatano]{Department of Mathematics, Faculty of Science and Technology, Keio University, 3-14-1 Hiyoshi, Kohoku-ku, Yokohama 223-8522, Japan/Center for Advanced Intelligence Project, RIKEN, Japan, and Department of Mathematics, Chuo University, 1-13-27, Kasuga, Bunkyo-ku, Tokyo 112-8551, Japan,}
\address[Masahiro Ikeda]{Center for Advanced Intelligence Project, RIKEN, Japan/Department of Mathematics, Faculty of Science and Technology, Keio University, 3-14-1 Hiyoshi, Kohoku-ku, Yokohama 223-8522, Japan
}
\address[Isao Ishikawa]{Center for Advanced Intelligence Project, RIKEN, Japan, and Department of Engineering for Production and Environment, Graduate School of Science and Engineering, Ehime University, 3 Bunkyo-cho, Matsuyama, Ehime 790-8577, Japan,}
\address[Yoshihiro Sawano]{Department of Mathematics, Faculty of Science and Technology, Keio University, 3-14-1 Hiyoshi, Kohoku-ku, Yokohama 223-8522, Japan/Center for Advanced Intelligence Project, RIKEN, Japan, and Department of Mathematics, Chuo University, 1-13-27, Kasuga, Bunkyo-ku, Tokyo 112-8551, Japan}





\maketitle

\begin{abstract}
In this study, we investigate the boundedness of composition operators 
acting on Morrey spaces and weak Morrey spaces.  
The primary aim of this study is to investigate
a necessary and sufficient condition on the boundedness 
of the composition operator induced by a diffeomorphism on Morrey spaces.
In particular, detailed information is derived from the boundedness, 
i.e.,
 the bi-Lipschitz continuity of the mapping that induces the composition operator
follows from the continuity of the composition mapping. 
The idea of the proof is to determine the Morrey norm of the characteristic functions, 
and employ a specific function composed of a characteristic function. 
As the specific function belongs to Morrey spaces 
but not to Lebesgue spaces, 
 the result reveals a new phenomenon
 not observed in Lebesgue spaces.
Subsequently, we prove the boundedness of the composition operator 
induced by a mapping that satisfies a suitable volume estimate 
on general weak-type spaces generated by normed spaces. 
As a corollary, a necessary and sufficient condition 
for the boundedness of the composition operator 
on weak Morrey spaces is provided.

\end{abstract}

\section{Introduction}

In this study, 
we investigate the boundedness of composition operators 
on Morrey spaces and weak Morrey spaces. 
The composition operator $C_{\varphi}$ induced by a mapping $\varphi$ is a linear operator defined by $C_{\varphi}f\equiv f\circ\varphi$, where $f\circ\varphi$ 
represents the function composition. 
The composition operator is also called 
the Koopman operator in the fields of dynamical systems, physics, 
and engineering \cite{Koopman31}. Recently, it has attracted attention in various scientific fields \cite{Kawahara, IFIHK}.

Let $(X,\mu)$ be a $\sigma$-finite measure space, and $L^0(X,\mu)$ be the set of all $\mu$-measurable functions on $X$.  We provide a precise definition of the composition operators induced by a measurable map $\varphi:X\rightarrow X$.
\begin{definition}[Composition operator]
 Let $\varphi:X\rightarrow X$ be a measurable map, and assume that $\varphi$ is nonsingular, namely, $\mu(\varphi^{-1}(E))=0$ for each $\mu$-measurable null set $E$. Let $V$ and $W$ be function spaces contained in $L^0(X,\mu)$. The {\em composition operator} $C_{\varphi}$ is the linear operator from $W$ to $V$ such that its domain is $\mathcal{D}(C_{\varphi})\equiv \{h \in W : h\circ \varphi\in V\}$, and $C_{\varphi}f\equiv f\circ \varphi$ for $f\in \mathcal{D}(C_{\varphi})$.
\end{definition}

Subsequently, we employ the result obtained by Singh \cite{Singh76}
for the boundedness of the composition operator 
on the Lebesgue space $L^p(X,\mu)$. 
Henceforth, 
 we denote by $L^0(X,\mu)$ the space of all $\mu$-measurable functions.

Singh \cite{Singh76} provided the following necessary and sufficient condition 
for the map $\varphi$ to generate a bounded mapping
acting on Lebesgue spaces:

\begin{theorem}[\cite{Singh76}]\label{thm:200727-1}
\label{thm singh}
Let $0<p<\infty$. Then, the composition operator $C_\varphi$ induced by $\varphi:X\rightarrow X$ 
is
bounded on the Lebesgue space $L^p(X,\mu)$ if and only 
if there exists a constant $K=K(\varphi)$ such that for all $\mu$-measurable sets $E$ in ${\mathbb R}^n$,
\begin{equation*}
\mu(\varphi^{-1}(E))\le K\mu(E).
\end{equation*}
In this case, the operator norm is given by
\begin{equation}\label{eq:200216-2.2}
\|C_\varphi\|_{L^p\to L^p}
=
\sup_{0<\mu(E)<\infty}\left(\frac{\mu(\varphi^{-1}(E))}{\mu(E)}\right)^{\frac1p}.
\end{equation}
\end{theorem}
The
boundedness of the composition operator
on $L^{\infty}(X,\mu)$ easily follows from the definition. 
Theorem \ref{thm:200727-1} was extended to several important function spaces, such as Lorentz spaces \cite{ADV07, Evseev17}, Orlicz spaces \cite{CHRL04, EvMe19-1,Kumar97}, 
mixed Lebesgue spaces \cite{EvMe19-2}, 
Musielak-Orlicz spaces \cite{RaSh12} 
and reproducing kernel Hilbert spaces \cite{IISpre}. However, there are no previous results 
on the boundedness of composition operators 
acting on Morrey spaces and weak Morrey spaces. 

The first aim of this strudy is to investigate a necessary and sufficient condition on the boundedness of the composition operator $C_\varphi$
on Morrey spaces.
Subsequently, we discuss
 the boundedness of the operator on weak Morrey spaces.

Hereafter, we consider the Euclidean space $\mathbb{R}^n$;
$\mu$ is the Lebesgue measure ${\rm dx}$. 
The set of all measurable functions is denoted
by $L^0({\mathbb R}^n)$.
We denote by $|E|$ the volume of a measurable set $E \subset \mathbb{R}^n$. 
Let $\chi_A:\mathbb{R}^n\rightarrow\mathbb{R}_{\ge 0}$ be an indicator function for a subset $A\subset \mathbb{R}^n$, 
which is defined as $\chi_A(x)=1$ 
if $x\in A$ and $\chi_A(x)=0$,
otherwise.

Now, we recall the definition of Morrey spaces on $\mathbb{R}^n$.
\begin{definition}[Morrey space]
Let $0<q\le p<\infty$. The {\it Morrey space} ${\mathcal M}^p_q({\mathbb R}^n)$ is a quasi-Banach space defined by
\begin{equation*}
{\mathcal M}^p_q({\mathbb R}^n)
\equiv
\{
f\in L^0({\mathbb R}^n)
:
\|f\|_{{\mathcal M}^p_q}
<\infty
\},
\end{equation*}
endowed with the quasi-norm
\[
   \|f\|_{{\mathcal M}^p_q}\equiv
\sup_{Q\in{\mathcal Q}}|Q|^{\frac1p-\frac1q}\|f\chi_Q\|_{L^q},
\]
where ${\mathcal Q}$ denotes the family of all cubes parallel to the coordinate axis in ${\mathbb R}^n$.
\end{definition}

From the H\"older inequality, 
we observe that the Lebesgue space
$L^p({\mathbb R}^n)$ is embedded into the Morrey space
${\mathcal M}^p_q({\mathbb R}^n)$,
 where $0< q \le p<\infty$.

\begin{remark}\label{rem:190809-1}
Let $0<q\le p<\infty$.
Then, we have
\begin{equation*}
L^p({\mathbb R}^n)
=
{\mathcal M}^p_p({\mathbb R}^n)
\subset
{\mathcal M}^p_q({\mathbb R}^n).
\end{equation*}
Moreover, $L^p({\mathbb R}^n)$ is not dense in ${\mathcal M}^p_q({\mathbb R}^n)$ \cite{Sawano17}.
\end{remark}

We now state the main results of the present paper.
The following theorem provides a sufficient condition 
on the boundedness of the composition operator 
$C_{\varphi}$ on the Morrey space $\mathcal{M}^p_q(\mathbb{R}^n)$.

\begin{theorem}\label{thm:200319-1.1}
Let $0<q\le p<\infty$.
Then, the composition operator $C_{\varphi}$ 
induced by $\varphi:{\mathbb R}^n\to{\mathbb R}^n$ is bounded 
on the Morrey space ${\mathcal M}^p_q({\mathbb R}^n)$, 
 if $\varphi$ is a Lipschitz map that satisfies the volume estimate
\begin{equation}\label{eq:200531-1}
|\varphi^{-1}(E)|\le K|E|,
\end{equation}
for all measurable sets $E$ in ${\mathbb R}^n$ and some constant $k$ independent of $E$.
In particular, we obtain
\begin{equation}\label{eq:200701-1}
\|C_\varphi\|_{{\mathcal M}^p_q\to{\mathcal M}^p_q}
\le
(\max(1,\sqrt{n}L))^{-\frac np+\frac nq}
\sup_{E:0<|E|<\infty}\left(\frac{|\varphi^{-1}(E)|}{|E|}\right)^{\frac1q},
\end{equation}
where $L>0$ is a Lipschitz constant of $\varphi$,
 and $E$ in the supremum moves over all the
Lebesgue measurable sets $E$
satisfying $0<|E|<\infty$.
\end{theorem}

Conversely, as stated in the following theorem, if $\varphi:\mathbb{R}^n\rightarrow\mathbb{R}^n$ is a diffeomorphism, then 
the ${\mathcal M}^p_q({\mathbb R}^n)$-boundedness 
of the composition operators $C_\varphi$ and $C_{\varphi^{-1}}$ indicates that $\varphi$ is bi-Lipschitz.
Note that any bi-Lipschitz mapping satisfies the assumption of Theorem \ref{thm:200319-1.1}. 

\begin{theorem}\label{thm:200319-1.3}
Let $n\in \mathbb{N}$, and $\varphi:\mathbb{R}^n\rightarrow\mathbb{R}^n$ be a diffeomorphism 
in the sense that $\varphi$ and its inverse $\varphi^{-1}$ are differentiable. 
Suppose $0<q<p<\infty$, or $q=p$ and $n=1$. If the composition operators $C_{\varphi}$ and $C_{\varphi^{-1}}$ induced by maps $\varphi$ and $\varphi^{-1}$, respectively, are bounded on ${\mathcal M}^p_q({\mathbb R}^n)$, then $\varphi$ is bi-Lipschitz.
\end{theorem}

\begin{remark}
In the case of $p=q$ and $n=1$, 
Theorem \ref{thm:200319-1.3}
reduced to Theorem \ref{thm singh} 
according to \cite{Singh76}.
Unless $n=1$, condition $q<p$ is essential in the following sense. If $n\ge 2$ and $q=p$, then the same conclusion as in Theorem \ref{thm:200319-1.3} fails. We present a counterexample in Example \ref{ex:200603-1} in Section \ref{s4}. 
Noting that the Morrey space ${\mathcal M}^p_p({\mathbb R}^n)$ 
coincides with the Lebesgue space $L^p({\mathbb R}^n)$ 
(see Remark \ref{rem:190809-1}), we observe a new phenomenon from Theorem
 \ref{thm:200319-1.3}.
\end{remark}

Subsequently, we investigate the characterization 
of the boundedness of the composition operators
acting on weak Morrey spaces, 
which are defined as follows:
\begin{definition}[Weak Morrey space]
Let $0<q\le p<\infty$. The {\it weak Morrey space} ${\rm W}{\mathcal M}^p_q({\mathbb R}^n)$ is a quasi-Banach space defined by
\begin{equation*}
{\rm W}{\mathcal M}^p_q({\mathbb R}^n)
\equiv
\{
f\in L^0({\mathbb R}^n)
:
\|f\|_{{\rm W}{\mathcal M}^p_q}
<\infty
\}
\end{equation*}
endowed with the quasi-norm
\[
   \|f\|_{{\rm W}{\mathcal M}^p_q}\equiv
\sup_{\lambda>0}\lambda\|\chi_{\{x\in{\mathbb R}^n:|f(x)|>\lambda\}}\|_{{\mathcal M}^p_q}.
\]
\end{definition}

The
weak Morrey space
${\rm W}{\mathcal M}^p_q({\mathbb R}^n)$ 
has
the following basic properties:

\begin{remark}\label{rem:190809-1}
Let $0<q<p<\infty$.
Then, we have
\begin{equation*}
{\mathcal M}^p_q({\mathbb R}^n)
\subset
{\rm W}{\mathcal M}^p_q({\mathbb R}^n),
\quad
{\rm W}{\mathcal M}^p_p({\mathbb R}^n)
=
{\rm W}L^p({\mathbb R}^n),
\end{equation*}
where ${\rm W}L^p({\mathbb R}^n)$ is the weak Lebesgue space (see \cite[Chapter 1]{Grafakos14-1} for more).
\end{remark}

The following theorem provides a necessary and sufficient condition 
on the boundedness of the composition operator on weak Morrey spaces.

\begin{theorem}\label{thm:200303-1}
Let $0<q\le p<\infty$, and let $\varphi:{\mathbb R}^n\to{\mathbb R}^n$ be 
a measurable function.
Then, 
$\varphi$
generates
the composition operator $C_\varphi$ which is bounded 
on the weak Morrey space ${\rm W}{\mathcal M}^p_q({\mathbb R}^n)$ if and only if there exists a constant $K$ such that for all measurable sets $E$ in ${\mathbb R}^n$, the estimate
\begin{equation*}
\|\chi_{\varphi^{-1}(E)}\|_{{\mathcal M}^p_q}
\le K
\|\chi_E\|_{{\mathcal M}^p_q}.
\end{equation*}
holds.
In particular, we obtain
\begin{equation*}
\|C_\varphi\|_{{\rm W}{\mathcal M}^p_q\to {\rm W}{\mathcal M}^p_q}
=
\sup_E\frac{\|\chi_{\varphi^{-1}(E)}\|_{{\mathcal M}^p_q}}{\|\chi_E\|_{{\mathcal M}^p_q}},
\end{equation*}
where the supremum is taken over all the measurable sets $E$ in ${\mathbb R}^n$ with $0<\|\chi_E\|_{{\mathcal M}^p_q}<\infty$.
\end{theorem}

\begin{remark}
\begin{itemize}
\item[(1)] Theorem \ref{thm:200303-1} 
indicates that the composition operator $C_\varphi$
is bounded 
on weak Morrey spaces,
 once it is bounded
on Morrey spaces (see Section \ref{s4} for more).
\item[(2)] The conclusion of cases $q=p$ in this theorem was provided in \cite{CVFR15}.
\item[(3)] Theorem \ref{thm:200303-1} is a special case of Theorem \ref{thm:200503-1} below.
\end{itemize}
\end{remark}

In fact, we 
will establish
the boundedness of the composition operator 
in a more general framework.

\begin{definition}
Let $(B({\mathbb R}^n),\|\cdot\|_B)$ be a linear subspace of $L^0({\mathbb R}^n)$ such that $\||f|\|_B=\|f\|_B$ for all $f\in B({\mathbb R}^n)$.
The weak-type space $({\rm W}B({\mathbb R}^n),\|\cdot\|_{{\rm W}B})$ of $B$ is defined by
\begin{equation*}
{\rm W}B({\mathbb R}^n)
\equiv
\{
f\in L^0({\mathbb R}^n)
:
\|f\|_{{\rm W}B}
<\infty
\},
\end{equation*}
endowed with the quasi-norm
\[
\|f\|_{{\rm W}B}
\equiv
\sup_{\lambda>0}\lambda\|\chi_{\{x\in{\mathbb R}^n:|f(x)|>\lambda\}}\|_B.
\]
\end{definition}

Now, we can rewrite Theorem \ref{thm:200303-1} as follows:

\begin{theorem}\label{thm:200503-1}
Let $(B({\mathbb R}^n),\|\cdot\|_B)$ be a normed space.
Then, the composition induced by $\varphi$ is bounded on the weak-type space $({\rm W}B({\mathbb R}^n),\|\cdot\|_{{\rm W}B})$ if and only if there exists a constant $K$ such that for all measurable sets $E$ in ${\mathbb R}^n$, the estimate
\begin{equation}\label{eq:200503-1}
\|\chi_{\varphi^{-1}(E)}\|_B
\le K
\|\chi_E\|_B.
\end{equation}
holds.
In particular, we obtain
\begin{equation}\label{eq:200720-1}
\|C_\varphi\|_{{\rm W}B\to {\rm W}B}
=
\sup_E\frac{\|\chi_{\varphi^{-1}(E)}\|_B}{\|\chi_E\|_B},
\end{equation}
where the supremum is taken over all the measurable sets $E$ in ${\mathbb R}^n$ with $0<\|\chi_E\|_B<\infty$.
\end{theorem}

The remainder of this paper is organized as follows:
In Section \ref{s2}, we prove Theorems \ref{thm:200319-1.1} and \ref{thm:200319-1.3}.
In Section \ref{s3}, we present some examples and counterexamples 
of the mapping that induces the composition operator to be bounded on Morrey spaces.
In Section \ref{s4}, we prove Theorem \ref{thm:200503-1}.

\section{Proof of Theorems \ref{thm:200319-1.1} and \ref{thm:200319-1.3}}\label{s2}
In this section, we prove Theorems \ref{thm:200319-1.1} and \ref{thm:200319-1.3}. 
The proof of Theorem \ref{thm:200319-1.1} is provided in Subsection \ref{ss2.1}. However, Theorem \ref{thm:200319-1.3} is more difficult to prove.
In Subsection \ref{ss2.2}, we reduce 
matters
to the linear setting.
We divide its proof into two steps:
we consider case $p\le nq$ in Subsection \ref{ss2.3} 
and case $nq\le p$ in Subsection \ref{ss2.4}.

\subsection{Proof of Theorem \ref{thm:200319-1.1}}
\label{ss2.1}
\begin{proof}[Proof of Theorem \ref{thm:200319-1.1}]
A cube, $Q\in \mathcal{Q}$, is fixed.
We note that, according to the Lipschitz continuity of $\varphi$, the estimates
\begin{align*}
{\rm diam}(\varphi(Q))
:=
\sup_{x,\tilde{x}\in Q}|\varphi(x)-\varphi(\tilde{x})|
\le
L\sup_{x,\tilde{x}\in Q}|x-\tilde{x}|
=
\sqrt{n}L\ell(Q),
\end{align*}
hold; thus, there exist cubes $Q_1, Q_2\in \mathcal{Q}$ such that
\begin{equation*}
Q_1\supset Q, \quad
Q_2\supset \varphi(Q), \quad
|Q_1|=|Q_2|=(\max(1,\sqrt{n}L))^n|Q|.
\end{equation*}
As $\varphi$ satisfies 
condition (\ref{eq:200531-1}), 
we can apply the $L^q({\mathbb R}^n)$-boundedness 
of the composition operators (Theorem \ref{thm:200727-1}) to obtain
\begin{align*}
&|Q|^{\frac1p-\frac1q}\left(\int_Q|f(\varphi(x))|^q\,{\rm d}x\right)^{\frac1q}
\le
|Q|^{\frac1p-\frac1q}
\left(
\int_{{\mathbb R}^n}|f(\varphi(x))|^q\chi_{\varphi(Q)}(\varphi(x))\,{\rm d}x
\right)^{\frac1q}\\
&\le
|Q|^{\frac1p-\frac1q}
\cdot\|C_\varphi\|_{L^q\to L^q}
\left(\int_{{\mathbb R}^n}|f(x)|^q\chi_{\varphi(Q)}(x)\,{\rm d}x\right)^{\frac1q}\\
&\le
((\max(1,\sqrt{n}L))^{-n}|Q_1|)^{\frac1p-\frac1q}
\cdot\|C_\varphi\|_{L^q\to L^q}
\left(\int_{{\mathbb R}^n}|f(x)|^q\chi_{Q_2}(x)\,{\rm d}x\right)^{\frac1q}\\
&\le
(\max(1,\sqrt{n}L))^{-\frac np+\frac nq}\|C_\varphi\|_{L^q\to L^q}
\|f\|_{{\mathcal M}^p_q},
\end{align*}
which indicates that the composition operator $C_{\varphi}$ is bounded on $\mathcal{M}^p_q(\mathbb{R}^n)$. Moreover, by applying the equation (\ref{eq:200216-2.2}), we obtain (\ref{eq:200701-1}), which completes the proof of the theorem.
\end{proof}
\subsection{Reduction of the diffeomorphism to the linear setting}\label{ss2.2}
In the following, for a differentiable vector-valued function $\varphi=(\varphi_1,\ldots,\varphi_n)^{\rm T}$ on $\mathbb{R}^n$, we denote by $D\varphi$ the Jacobi matrix of $\varphi$, that is,
\begin{equation*}
D\varphi
\equiv
\left(\frac{\partial\varphi_i}{\partial x_j}\right)_{1\le i,j\le n}
\equiv
(\varphi_{i,j})_{1\le i,j\le n}.
\end{equation*}
In this subsection, by applying the following lemma (Lemma \ref{lem:200506-1}), we reduce the diffeomorphism $\varphi:\mathbb{R}^n\rightarrow\mathbb{R}^n$ in Theorem \ref{thm:200319-1.3} to the linear mapping $D\varphi:\mathbb{R}^n\rightarrow M_n(\mathbb{R})$. By the estimate of the singular values of the Jacobi matrix $D\varphi$, we 
will
show that $\varphi$ is bi-Lipschitz (see Proposition \ref{prop:200722-1} below).

\begin{lemma}\label{lem:200506-1}
Let $0<q\le p<\infty$.
Suppose that a diffeomorphism
$\varphi:{\mathbb R}^n \to {\mathbb R}^n$
induces a bounded composition operator $C_{\varphi}$ from
${\mathcal M}^p_q({\mathbb R}^n)$ to itself.
Then, there exists a positive constant $k>0$ such that
\[
\|C_{D\varphi(x_0)}f\|_{{\mathcal M}^p_q}=\|f(D\varphi(x_0)\cdot)\|_{{\mathcal M}^p_q}
\le K
\|f\|_{{\mathcal M}^p_q}
\]
for all $x_0 \in {\mathbb R}^n$
and
$f \in {\mathcal M}^p_q({\mathbb R}^n)$. In particular, the operator norm of $\|C_{D\varphi(x_0)}\|$ is bounded above by a constant independent of $x_0$.
\end{lemma}

\begin{proof}[Proof of Lemma \ref{lem:200506-1}]
Set $K\equiv\|C_\varphi\|_{{\mathcal M}^p_q\to{\mathcal M}^p_q}<\infty$.
First, we prove the assertion
for $f \in C^\infty_{\rm c}({\mathbb R}^n)$, where $C^\infty_{\rm c}({\mathbb R}^n)$ is the set of all smooth functions with compact support.
Let $t>0$.
We calculate
\begin{align*}
\left\|f\left(\frac{\varphi(x_0+t \cdot)-\varphi(x_0)}{t}\right)\right\|_{{\mathcal M}^p_q}
&=t^{-\frac{n}{p}}
\left\|f\left(\frac{\varphi(\cdot)-\varphi(x_0)}{t}\right)\right\|_{{\mathcal M}^p_q}\\
&\le K
t^{-\frac{n}{p}}
\left\|f\left(\frac{\cdot-\varphi(x_0)}{t}\right)\right\|_{{\mathcal M}^p_q}\\
&=K
\left\|f\right\|_{{\mathcal M}^p_q}.
\end{align*}
By letting $t \to 0$, we obtain the desired result for $f \in C^\infty_{\rm c}({\mathbb R}^n)$.

Let
$f \in L^\infty_{\rm c}({\mathbb R}^n)$, where $L^\infty_{\rm c}({\mathbb R}^n)$ is the set of all $L^\infty({\mathbb R}^n)$-functions with compact support.
Then, for any $p\in (0,\infty)$, we can choose a sequence
$\{f_j\}_{j=1}^\infty$
of
compactly supported smooth functions
such that
$f_j$ converges to $f$
in $L^p({\mathbb R}^n)$ as $j\rightarrow\infty$.
By passing to a subsequence,
we may assume that
$f_j$ converges to $f$,
almost everywhere in $\mathbb{R}^n$ as $j\rightarrow\infty$. Thus, by the Fatou lemma, the inequality
\[
\|f(D\varphi(x_0)\cdot)\|_{{\mathcal M}^p_q}
\le
\liminf_{j \to \infty}
\|f_j(D\varphi(x_0)\cdot)\|_{{\mathcal M}^p_q}
\]
holds. As we have proved the assertion for
$f_j$, we have
\[
\|f_j(D\varphi(x_0)\cdot)\|_{{\mathcal M}^p_q}
\le K
\|f_j\|_{{\mathcal M}^p_q}.
\]
As $L^p({\mathbb R}^n)$ is embedded into
${\mathcal M}^p_q({\mathbb R}^n)$ (see Remark \ref{rem:190809-1}), $f_j$ converges to $f$
in 
${\mathcal M}^p_q({\mathbb R}^n)$ as $j\rightarrow\infty$.
Consequently,
\[
\liminf_{j \to \infty}
\|f_j\|_{{\mathcal M}^p_q}
=
\|f\|_{{\mathcal M}^p_q}.
\]
By combining these observations, the following estimate holds:
\[
\|f(D\varphi(x_0)\cdot)\|_{{\mathcal M}^p_q}
\le K
\|f\|_{{\mathcal M}^p_q}.
\]

Finally, let $f \in {\mathcal M}^p_q({\mathbb R}^n)$. For $k\in \mathbb{N}$, we set
$f_k\equiv f\chi_{[-k,k]^n}\chi_{[0,k]}(|f|)\in L^{\infty}_c(\mathbb{R}^n)$.
Then, we have
\[
\|f_k(D\varphi(x_0)\cdot)\|_{{\mathcal M}^p_q}
\le K
\|f_k\|_{{\mathcal M}^p_q}
\le K
\|f\|_{{\mathcal M}^p_q}
\]
according to the previous paragraph.
By using the Fatou lemma again, we obtain
\[
\|f(D\varphi(x_0)\cdot)\|_{{\mathcal M}^p_q}
\le K
\|f\|_{{\mathcal M}^p_q},
\]
as required.
\end{proof}

Let $\varphi:\mathbb{R}^n\rightarrow\mathbb{R}^n$ be a diffeomorphism and $D\varphi:\mathbb{R}^n\rightarrow M_n(\mathbb{R})$ be the Jacobi matrix of $\varphi$. For $x_0\in \mathbb{R}^n$, the Jacobi matrix $D\varphi(x_0)$ can be decomposed by the Singular value decomposition (see Lemma \ref{lem:200805-1} below) as
\begin{equation}\label{eq:200801-1}
D\varphi(x_0)=UAV,
\end{equation}
where $A=A(x_0)={\rm diag}(\alpha_1(x_0),\dots,\alpha_n(x_0))$ is a diagonal matrix with having positive components satisfying $\alpha_1(x_0)\le\dots\le\alpha_n(x_0)$, and $U=U(x_0)$ and $V=V(x_0)$ are orthogonal matrices.

\begin{lemma}[Singular value decomposition]
\label{lem:200805-1}
Let $A$ be an $n\times n$ real regular matrix, and $\alpha_1,\ldots,\alpha_n>0$ be the singular values of $A$.
Then, there exist orthogonal matrices $U$ and $V$ such that
\begin{equation*}
UAV
=
{\rm diag}(\alpha_1,\ldots,\alpha_n).
\end{equation*}
\end{lemma}

Now, by the definition of the Morrey norm
$\|\cdot\|_{{\mathcal M}^p_q}$,
and a simple computation, we have the equivalence
\begin{align*}
n^{\frac np-\frac nq}\|C_{A(x_0)}\|_{{\mathcal M}^p_q\to{\mathcal M}^p_q}
\le
\|C_{D\varphi(x_0)}\|_{{\mathcal M}^p_q\to{\mathcal M}^p_q}
\le
n^{-\frac np+\frac nq}\| C_{A(x_0)}\|_{{\mathcal M}^p_q\to{\mathcal M}^p_q}.
\end{align*}
Here, the operator norms $\|\cdot\|_{{\mathcal M}^p_q\to{\mathcal M}^p_q}$ of the composition operators induced by the orthogonal matrices are bounded above by a constant independent of the selection of the rotation matrices.  Therefore, we have the following lemma:
\begin{lemma}
\label{lem:200806}
Let $0<q\le p<\infty$.
Suppose that a diffeomorphism
$\varphi:{\mathbb R}^n \to {\mathbb R}^n$,
induces a bounded composition operator $C_{\varphi}$ on
${\mathcal M}^p_q({\mathbb R}^n)$. Let $\alpha_1(x_0), \dots,\alpha_n(x_0)$ be the singular values of $D\varphi(x_0)$, and let us denote $A(x_0):={\rm diag}(\alpha_1(x_0),\dots,\alpha_n(x_0))$.  Then,  the operator norm of $C_{A(x_0)}$ is bounded above by a constant independent of $x_0$.
\end{lemma}

Hereafter, we use shorthand $A(x)\lesssim B(x)$ to denote estimate $A(x)\leq CB(x)$ with
some constant $C>0$ independent of $x$. Notation $A(x)\sim B(x)$ represents $A(x)\lesssim B$(x)
and $B(x)\lesssim A(x)$.

\begin{proposition}\label{prop:200605-1}
Let $0<q\le p<\infty$, $x_0\in \mathbb{R}^n$, and $\varphi:\mathbb{R}^n\rightarrow\mathbb{R}^n$ be diffeomorphism. 
If the composition operators $C_{\varphi}$ and $C_{\varphi^{-1}}$ induced by $\varphi$ and $\varphi^{-1}$, respectively are bounded on the Morrey space ${\mathcal M}^p_q({\mathbb R}^n)$, then we have
\begin{equation*}
\alpha_1(x_0)\times\cdots\times\alpha_n(x_0)\sim1,
\end{equation*}
where $\alpha_1(x_0),\dots,\alpha_n(x_0)$ are the singular values of $D\varphi(x_0)$.
\end{proposition}

This proposition can be proved by combining Lemma \ref{lem:200506-1} and Lemma \ref{lem:200603-1} below.

\begin{lemma}\label{lem:200603-1}
Let $0<q\le p<\infty$ and $\{a_1,\ldots,a_n\}\subset \mathbb{R}_{>0}$ be a positive sequence and set $D\equiv{\rm diag}(a_1,\ldots,a_n)$. Then, the following estimate holds:
\begin{equation*}
a_1\times\cdots\times a_n
\ge
\|C_D\|_{{\mathcal M}^p_q \to {\mathcal M}^p_q}^{-p}.
\end{equation*}
\end{lemma}

\begin{proof}
We introduce matrix $W\in M_n(\mathbb{R})$ corresponding to the transform
\[
(x_1,x_2,\ldots,x_n) \mapsto (x_2,x_3,\ldots,x_n,x_1).
\]
Then, by a simple computation, for any $k\in \{1,\cdots,n\}$, we observe that identities
$W^{-k}DW^k={\rm diag}(a_{n-k+1},a_{n-k+2},\ldots,a_n,a_1,a_2,\ldots,a_{n-k})$
and
\[\|C_{W^{-k}DW^k}\|_{{\mathcal M}^p_q \to
{\mathcal M}^p_q}=\|C_D\|_{{\mathcal M}^p_q\to{\mathcal M}^p_q}
\]
hold. Noting that the identity
\[
\prod_{k=1}^n W^{-k}DW^k=
a_1a_2\ldots a_n E
\]
holds, where $E\in M_n(\mathbb{R})$ denotes the identity matrix, the equality
\[
\left\|C_{\overset{n}{\underset{k=1}{\prod}}W^{-k}DW^k}\right\|_{{\mathcal M}^p_q \to {\mathcal M}^p_q}
=
(a_1a_2\times\ldots\times a_n)^{-\frac{n}{p}}
\]
holds. By combining this and identity $C_{\overset{n}{\underset{k=1}{\prod}}W^{-k}DW^k}=\overset{n}{\underset{k=1}{\prod}}C_{W^{-k}DW^k}$, the conclusion of this lemma is proved.
\end{proof}

To obtain the bi-Lipschitz continuity of $\varphi$, we use the following proposition, which is obtained using the mean value theorem.

\begin{proposition}\label{prop:200722-1}
Let $\varphi:\mathbb{R}^n\rightarrow\mathbb{R}^n$ be a diffeomorphism.  Let $\alpha_1(x_0)$ be a minimal singular value.
If there exists a positive constant $C>0$ independent of $x_0$ such that for all $x_0\in{\mathbb R}^n$,
\begin{equation}
\label{goal}
\alpha_1(x_0)\ge C,
\end{equation}
then the inverse function $\varphi^{-1}$ of $\varphi$ is Lipschitz.
\end{proposition}

\begin{proof}[Proof of Proposition \ref{prop:200722-1}]
$x,\tilde{x}\in{\mathbb R}^n$ are fixed.
As mapping $\varphi^{-1}$ is differentiable on the line segment between $x$ and $\tilde{x}$, by the mean value
theorem, we can consider point $x_0$ on the line segment between $x$ and $\tilde{x}$ and obtain
\begin{align*}
|\varphi^{-1}(x)-\varphi^{-1}(\tilde{x})|
=
\|D\varphi^{-1}(x_0)\|_F|x-\tilde{x}|,
\end{align*}
where the quantity $\|A\|_F$ is a Frobenius norm defined by $\sqrt{{\rm tr}(A^{\rm T}A)}$ for matrix $A$.
Now, using the decomposition (\ref{eq:200801-1}), we can calculate
\begin{align*}
\|D\varphi^{-1}(x_0)\|_F
=
\left(\sum_{j=1}^n\left|\frac1{\alpha_j(x_0)}\right|^2\right)^{\frac12}
\le
\frac{\sqrt{n}}C.
\end{align*}
Consequently, we obtain the Lipschitz continuity of $\varphi^{-1}$.
\end{proof}

According to this proposition, to obtain Theorem \ref{thm:200319-1.3}, it suffices to show that there exists a positive constant $C>0$ such that for each $x_0\in{\mathbb R}^n$, the estimate (\ref{goal}) holds. We divide the proof of (\ref{goal}) into the two cases $p\le nq$ and $nq\le p$.

\subsection{Proof of (\ref{goal}) in the case $p\le nq$}\label{ss2.3}

To obtain the estimate (\ref{goal}), we estimate the operator norm of the diagonal matrices $A(x_0)$ in the decomposition (\ref{eq:200801-1}) as follows using Lemma \ref{lem:200626-2} and Proposition \ref{prop:200626-2} below.

\begin{lemma}\label{lem:200626-2}
Let $n\ge m\ge2$, $0<q\le\dfrac nmq\le p\le\dfrac n{m-1}q<\infty$ and $1\le a_1\le\cdots\le a_{n-1}$.
Then, we have
\begin{align*}
\lefteqn{
\|\chi_{[0,1]\times[0,a_1]\times\cdots\times[0,a_{n-1}]}\|_{{\mathcal M}^p_q}
}\\
\nonumber
&=
a_1^{\frac1q}\cdots a_{m-1}^{\frac1q}a_{m-1}^{\frac np-\frac mq}.
\end{align*}
\end{lemma}

\begin{proof}
As we have to consider only cubes
of form $[0,R]^n$ for $R>0$
as the candidates for supremum
in the Morrey norm $\|\cdot\|_{\mathcal{M}^p_q}$, we have identity
\[
\|\chi_{[0,1]\times[0,a_1]\times\cdots\times[0,a_{n-1}]}\|_{{\mathcal M}^p_q}
=
\sup_{R>0}
R^{\frac{n}{p}-\frac{n}{q}}\{
\min(1,R)\min(a_1,R)\cdots\min(a_{n-1},R)\}^{\frac1q}
\]
By considering the case of $R=1,a_1,\ldots,a_{n-1}$, we can determine the supremum on the right-hand side of the above identity as follows:
\begin{equation}\label{eq:200801-2}
\|\chi_{[0,1]\times[0,a_1]\times\cdots\times[0,a_{n-1}]}\|_{{\mathcal M}^p_q}
=
\max\left(1,a_1^{\frac np-\frac1q},a_1^{\frac1q}a_2^{\frac np-\frac2q},\ldots,a_1^{\frac1q}\cdots a_{n-2}^{\frac1q}a_{n-1}^{\frac np-\frac{n-1}{q}}\right).
\end{equation}

Here, according to assumption $p\le\dfrac n{m-1}q$, we observe that
\begin{align*}
a_1^{\frac1q}\cdots a_{m-1}^{\frac1q}a_{m-1}^{\frac np-\frac mq}
\ge
a_1^{\frac1q}\cdots a_{m-2}^{\frac1q}a_{m-2}^{\frac np-\frac{m-1}q}
\ge
\cdots
\ge
a_1^{\frac np-\frac1q}
\ge
1.
\end{align*}
According to the assumption $\dfrac nmq\le p$,
we calculate
\begin{align*}
a_1^{\frac1q}\cdots a_{m-1}^{\frac1q}a_{m-1}^{\frac np-\frac mq}
\ge
a_1^{\frac1q}\cdots a_m^{\frac1q}a_m^{\frac np-\frac{m+1}q}
\ge
\cdots
\ge
a_1^{\frac1q}\cdots a_{n-1}^{\frac1q}a_{n-1}^{\frac np-\frac nq},
\end{align*}
to conclude that quantity $\displaystyle a_1^{\frac1q}\cdots a_{m-1}^{\frac1q}a_{m-1}^{\frac np-\frac mq}$ is the largest when taking the maximum in the equation (\ref{eq:200801-2}).
Hence, this is the desired result.
\end{proof}

\begin{proposition}\label{prop:200626-2}
Let $n\ge m\ge2$, $0<q\le\dfrac nmq\le p\le\dfrac n{m-1}q<\infty$ and $1\le a_1\le\cdots\le a_{n-1}$.
Then, we have
\[
\|C_{{\rm diag}(1,a_1,\ldots,a_{n-1})}\|_{{\mathcal M}^p_q \to {\mathcal M}^p_q}
\ge
a_1^{-\frac1q}\cdots a_{m-1}^{-\frac1q}a_{m-1}^{-\frac np+\frac mq}.
\]
\end{proposition}

\begin{proof}
Them, we use
\[
\|\chi_{[0,1]\times[0,R_1]\times\cdots\times[0,R_{n-1}]}\|_{{\mathcal M}^p_q}
=
R_1^{\frac1q}\cdots R_{m-1}^{\frac1q}R_{m-1}^{\frac np-\frac mq},
\]
and
\[
\|\chi_{[0,1]\times[0,a_1^{-1}R_1]\times\cdots\times[0,a_{n-1}^{-1}R_{n-1}]}\|_{{\mathcal M}^p_q}
=
(a_1^{-1}R_1)^{\frac1q}\cdots(a_{m-1}^{-1}R_{m-1})^{\frac1q}(a_{m-1}^{-1}R_{m-1})^{\frac np-\frac mq}
\]
for $1\le R_1\le\cdots\le R_{n-1}$ with $1\le a_1^{-1}R_1\le\cdots\le a_{n-1}^{-1}R_{n-1}$.
\end{proof}

Now, we prove the estimate (\ref{goal}).
It suffices to consider case $\dfrac nmq<p\le\dfrac n{m-1}q$, for each $m=2,\ldots,n$.
According to 
Lemma \ref{lem:200806} and Proposition \ref{prop:200626-2}, we calculate
\begin{align}\label{eq:200626-5}
1
&\gtrsim
\|C_{A(x_0)}\|_{{\mathcal M}^p_q\to{\mathcal M}^p_q}\\
&\ge
\alpha_1(x_0)^{-\frac np}
\prod_{i\in I}
\left(\frac{\alpha_i(x_0)}{\alpha_1(x_0)}\right)^{-\frac1q}
\cdot
\left(
\frac{\alpha_j(x_0)}{\alpha_1(x_0)}
\right)^{-\frac np+\frac mq} \notag\\
&=
\left(\alpha_1(x_0)\prod_{i\in I}\alpha_i(x_0)\right)^{-\frac1q}
\alpha_j(x_0)^{-\frac np+\frac mq}, \notag
\end{align}
where $I$ is a subset of $\{2,\ldots,n\}$ such that $\sharp I=m-1$ and $j\in I$.
Combining the estmate (\ref{eq:200626-5}) and Proposition \ref{prop:200605-1}, we have
\begin{align*}
1
&\gtrsim
\prod_{\substack{I\subset\{2,\ldots,n\}\\ \sharp I=m-1}}\prod_{j\in I}
\left(
\left(\alpha_1(x_0)\prod_{i\in I}\alpha_i(x_0)\right)^{-\frac1q}
\alpha_j(x_0)^{-\frac np+\frac mq}
\right)\\
&=
\alpha_1(x_0)^{-\frac{m-1}q\binom{n-1}{m-1}}
(
\alpha_2(x_0)\cdots\alpha_n(x_0)
)^{\left(-\frac np+\frac1q\right)\binom{n-2}{m-2}}\\
&\sim
\alpha_1(x_0)^{
-\frac{m-1}q\binom{n-1}{m-1}+\left(\frac np-\frac1q\right)\binom{n-2}{m-2}
}
=
\alpha_1(x_0)^{\left(\frac np-\frac nq\right)\binom{n-2}{m-2}}
\end{align*}
and then
\begin{align*}
\alpha_1(x_0)\gtrsim1.
\end{align*}

\subsection{Proof of (\ref{goal}) in the case $nq\le p$}\label{ss2.4}

Let $nq\le p$.
Using Lemmas \ref{lem:200806} and \ref{lem:200603-3}, we 
obtain the estimate (\ref{goal}).
To prove Lemma \ref{lem:200603-3}, 
we will use Lemma \ref{lem:200603-2} below.

\begin{lemma}\label{lem:200603-2}
Let $0<q<n q \le p<\infty$.
Then,
$
\chi_{[0,1]\times{\mathbb R}^{n-1}} \in {\mathcal M}^p_q({\mathbb R}^n).
$
\end{lemma}

\begin{proof}
We calculate
\[
\|\chi_{[0,1]\times{\mathbb R}^{n-1}}\|_{{\mathcal M}^p_q}
=
\sup_{R>0}
R^{\frac{n}{p}-\frac{n}{q}}\min(1,R)^{\frac1q}R^{\frac{n-1}{q}}
=
\sup_{R>0}
R^{\frac{n}{p}-\frac{1}{q}}\min(1,R)^{\frac1q}=1<\infty.
\]
\end{proof}

\begin{lemma}\label{lem:200603-3}
Let $0<a_1\le\cdots\le a_n$.
We assume that
$D={\rm diag}(a_1,a_2,\ldots,a_n)$
induces a bounded composition operator on the Morrey space 
${\mathcal M}^p_q({\mathbb R}^n)$ with the operator norm $M$.
Moreover, we assume that
$0<q<n q \le p<\infty$.
Then,
$a_1 \ge M^{-\frac{p}{n}}$.
\end{lemma}

\begin{proof}
Note that
$\chi_{[0,1]\times{\mathbb R}^{n-1}} \circ (a_1 E)=
\chi_{[0,a_1{}^{-1}]\times{\mathbb R}^{n-1}}=
\chi_{[0,1]\times{\mathbb R}^{n-1}} \circ D$.
Using scaling, we calculate
\begin{align*}
a_1{}^{-\frac{n}{p}}
\|\chi_{[0,1]\times{\mathbb R}^{n-1}}\|_{{\mathcal M}^p_q}
&=
\|\chi_{[0,1]\times{\mathbb R}^{n-1}} \circ (a_n E)\|_{{\mathcal M}^p_q}\\
&=
\|\chi_{[0,1]\times{\mathbb R}^{n-1}} \circ D\|_{{\mathcal M}^p_q}\\
&\le M
\|\chi_{[0,1]\times{\mathbb R}^{n-1}}\|_{{\mathcal M}^p_q}.
\end{align*}
Thus, according to Lemma \ref{lem:200603-2}, this is the desired result.
\end{proof}

\section{Examples}\label{s3}

In this section, we present some examples and counterexamples.
In Example \ref{ex:200312-1}, the mapping inducing the composition operator satisfies the assumption in Theorem \ref{thm:200319-1.1}.
In Example \ref{ex:200415-1}, the mapping inducing the composition operator does not satisfy the assumption in Theorem \ref{thm:200319-1.1}; however, the composition operator is bounded on the Morrey spaces.
Example \ref{ex:200603-1} presents a counterexample of cases $n\ge2$ and $q=p$ in Theorem \ref{thm:200319-1.3}.

\begin{example}\label{ex:200312-1}
The affine map $\varphi$, written as $\varphi(x)=Ax+b$ for some $A\in{\rm GL}(n;{\mathbb R})$ and $b\in{\mathbb R}^n$ induces the composition operator $C_\varphi$ bounded on the Morrey space ${\mathcal M}^p_q({\mathbb R}^n)$
whenever $0<q\le p<\infty$.
This follows from the fact that mapping $\varphi$ satisfies the assumption of Theorem \ref{thm:200319-1.1}.
\end{example}

\begin{example}\label{ex:200415-1}
Let $n=1$ and $1<p<\infty$.
Then, the composition operator induced by $\varphi:{\mathbb R}\to{\mathbb R}$,
\begin{equation*}
\varphi(x)
\equiv
\begin{cases}
e^x-1, & \mbox{if $x\ge0$},\\
x, & \mbox{if $x<0$},
\end{cases}
\end{equation*}
is bounded on ${\mathcal M}^p_1({\mathbb R})$ and $\varphi$ satisfies the volume estimate (\ref{eq:200531-1}); however, $\varphi$ is not Lipschitz.

Here, we prove that the composition mapping $C_{\varphi}$ induced by $\varphi:{\mathbb R}\to{\mathbb R}$ is bounded on ${\mathcal M}^p_1({\mathbb R})$.
It suffices to show that, for all $a\ge0$ and $b>0$,
\begin{equation}\label{eq:200216-3}
(b-a)^{\frac1p-1}\int_a^b|C_\varphi f(x)|\,{\rm d}x
\lesssim
\|f\|_{{\mathcal M}^p_1}
\end{equation}

Now, we check inequality (\ref{eq:200216-3}).
If $0<b-a\le1$, through the change of variables as $y=e^x-1$ and the fact that $e^b-e^a\sim e^a(b-a)$, we obtain
\begin{align*}
&(b-a)^{\frac1p-1}\int_a^b|C_\varphi f(x)|\,{\rm d}x
=
(b-a)^{\frac1p-1}\int_{e^a-1}^{e^b-1}|f(y)|\,\frac{{\rm d}y}{y+1}\\
&\le
\{e^a(b-a)\}^{\frac1p-1}\cdot e^{-\frac ap}\int_{e^a-1}^{e^b-1}|f(y)|\,{\rm d}y
\lesssim
\|f\|_{{\mathcal M}^p_1}.
\end{align*}
Furthermore, when $b-a>1$, or equivalently, $(b-a)^{-1}<1$, we calculate
\begin{align*}
&(b-a)^{\frac1p-1}\int_a^b|C_\varphi f(x)|\,{\rm d}x
\le
\sum_{j=0}^\infty\int_{2^j-1}^{2^{j+1}-1}|f(y)|\,\frac{{\rm d}y}{y+1}\\
&\le
\sum_{j=0}^\infty2^{-\frac jp}\cdot(2^j)^{\frac1p-1}\int_{2^j-1}^{2^{j+1}-1}|f(y)|\,{\rm d}y
\lesssim
\|f\|_{{\mathcal M}^p_1}
\end{align*}
as desired.
\end{example}

\begin{example}\label{ex:200603-1}
Let
\[
\varphi(x_1,x_2)\equiv\left(x_1{}^3+x_1,\dfrac{x_2}{3x_1{}^2+1}\right)
\]
be a diffeomorphism on ${\mathbb R}^2$. 
Let us consider the boundedness of $C_\varphi$ 
on ${\mathcal M}^p_q({\mathbb R}^2)$.
In the case of $p=q$, $C_\varphi$ is bounded, $D\varphi(x_1,x_2)$ has determinant $1$.
In contrast, in the case of $p>q$, $C_\varphi$ is not bounded; in fact, the first component is not Lipschitz.
\end{example}

\section{Boundedness of composition operators on weak type spaces}\label{s4}

To prove Theorem \ref{thm:200503-1}, we use the following identity.

\begin{remark}\label{rem:200503-2}
Through a simple calculation, we have $\|\chi_E\|_{{\rm W}B}=\|\chi_E\|_B$ for all measurable sets $E$ in ${\mathbb R}^n$.
\end{remark}

\begin{proof}
First, we assume that the composition operator $C_\varphi$ is bounded on ${\rm W}B({\mathbb R}^n)$, that is, there exists a constant $K$ such that the estimate
\[
\|C_\varphi f\|_{{\rm W}B}
\le K \|f\|_{{\rm W}B}
\]
holds for any $f\in {\rm W}B({\mathbb R}^n)$. Then, upon choosing $f=\chi_E$, the estimates
\begin{align*}
\|\chi_{\varphi^{-1}(E)}\|_{{\rm W}B}
=
\|C_\varphi\chi_E\|_{{\rm W}B}
\le K
\|\chi_E\|_{{\rm W}B}
\end{align*}
hold. By using Remark \ref{rem:200503-2},
we conclude that
\begin{align*}
\|\chi_{\varphi^{-1}(E)}\|_B
\le K
\|\chi_E\|_B.
\end{align*}
Second, we assume the condition (\ref{eq:200503-1}).
Considering $E=\{x\in{\mathbb R}^n:|f(x)|>\lambda\}$, we have
\begin{align*}
\|C_\varphi f\|_{{\rm W}B}
=
\sup_{\lambda>0}\lambda\|\chi_{\varphi^{-1}(\{x\in X:|f(x)|>\lambda\})}\|_B
\le K
\sup_{\lambda>0}\lambda\|\chi_{\{x\in X:|f(x)|>\lambda\}}\|_B
=
\|f\|_{{\rm W}B}.
\end{align*}

Finally, the equation
\begin{equation}\label{eq:200720-2}
\|C_\varphi\|_{{\rm W}B\to {\rm W}B}
\le
\sup_E\frac{\|\chi_{\varphi^{-1}(E)}\|_B}{\|\chi_E\|_B}
\end{equation}
is trivial.
According to the definition of the operator norm $\|\cdot\|_{{\rm W}B\to {\rm W}B}$,
\begin{equation}\label{eq:200720-3}
\|C_\varphi\|_{{\rm W}B\to {\rm W}B}
\ge
\sup_E\frac{\|\chi_{\varphi^{-1}(E)}\|_B}{\|\chi_E\|_B}
\end{equation}
Combining these estimates (\ref{eq:200720-2}) and (\ref{eq:200720-3}), we obtain equation (\ref{eq:200720-1}).
\end{proof}

The weak type spaces generated by the Banach lattice are essential.

\begin{definition}\label{def:Banach lattice}
We say that a Banach space $(B({\mathbb R}^n),\|\cdot\|_B)$ contained in $L^0({\mathbb R}^n)$ is a Banach lattice 
if the inequality $\|f\|_B\le \|g\|_B$ holds for all $f,g\in B({\mathbb R}^n)$ that satisfies $|f|\le|g|$, a.e..
\end{definition}

\begin{remark}
If $B({\mathbb R}^n)$ is a Banach lattice (see Definition \ref{def:Banach lattice}), then $B({\mathbb R}^n)$ is embedded in ${\rm W}B({\mathbb R}^n)$.
\end{remark}

Now, 
as the special case of the Morrey space $B({\mathbb R}^n)={\mathcal M}^p_q({\mathbb R}^n)$,
in Theorem \ref{thm:200503-1}, we have Theorem \ref{thm:200303-1}.

In Theorem \ref{thm:200303-1}, through real interpolation, it is known that
\begin{equation*}
{\rm W}{\mathcal M}^p_q({\mathbb R}^n)=[{\mathcal M}^{pr}_{qr}({\mathbb R}^n),L^\infty({\mathbb R}^n)]_{1-r,\infty}
\end{equation*}
(see \cite{CwGu00}).
Here, as the $L^\infty({\mathbb R}^n)$-boundedness of the composition operators is trivial and the ${\mathcal M}^p_q({\mathbb R}^n)$-boundedness 
and ${\mathcal M}^{pr}_{qr}({\mathbb R}^n)$-boundedness of composition operators, for $r>0$, are equivalent owing to the fact that $|C_\varphi f|^r=C_\varphi[|f|^r]$ for mapping $\varphi$, then we obtain that the boundedness $\lq\lq C_\varphi:{\mathcal M}^p_q({\mathbb R}^n)\to{\mathcal M}^p_q({\mathbb R}^n)$ implies $C_\varphi:{\rm W}{\mathcal M}^p_q({\mathbb R}^n)\to {\rm W}{\mathcal M}^p_q({\mathbb R}^n)$".

\vspace{10pt}
{\em Acknowledgement}. 
This work was supported by a JST CREST Grant (Number JPMJCR1913, Japan).
This work was also supported by the RIKEN Junior Research Associate Program.
The second author is supported by a Grant-in-Aid for Young Scientists Research 
(No.19K14581), Japan Society for the Promotion of Science.
The fourth author is supported by a 
Grant-in-Aid for Scientific Research (C) (19K03546), 
Japan Society for the Promotion of Science.

\end{document}